\documentclass[smallcondensed,natbib]{svjour3}

\usepackage{amsmath}
\usepackage{amsfonts}
\usepackage{graphicx}
\usepackage{color}
\usepackage{wrapfig}
\usepackage{hyperref}
\usepackage{epstopdf}
\usepackage{authblk}

\usepackage{xspace}
\newcommand{\eg}{{\em e.g.}\xspace}          
\newcommand{\partialx}[1]{\frac{\partial {#1}}{\partial x}}

\newcommand{\partialxx}[1]{\frac{\partial^2 {#1}}{\partial x^2}}

\newcommand{\fhn}{{FitzHugh-Nagumo}\xspace}    
\newcommand{\Jc}{{\mathcal{J}(c)}\xspace}    
\newcommand{\Lc}{{\mathcal{L}_c}\xspace}    
\newcommand{\MPM}{{-/+/-}\xspace}    
    
\newcommand{\A}{{\mathcal{A}}\xspace}    
\newcommand{\Q}{{\mathcal{Q}}\xspace}    
\newcommand{\F}{{\mathcal{F}}\xspace}    
\newcommand{\Lex}{{L^2_{ex}}\xspace}    
\newcommand{\Hex}{{H^1_{ex}}\xspace}

\newcommand{\ds}{\displaystyle}

\newcommand{\vectwo}[2]{ \left( \begin{array}{c} #1 \\ #2 \end{array} \right)}

\newcommand{\mattwo}[4]{\left( \begin{array}{cc} #1 & #2 \\ #3 & #4 \end{array}
\right)}

\newtheorem{lem}{Lemma}
\newtheorem{defn}{Definition}

\newcommand{\secref}[1]{Section~\ref{#1}}
\newcommand{\figref}[1]{Figure~\ref{#1}}
\newcommand{\defref}[1]{Definition~\ref{#1}}
\newcommand{\algref}[1]{Algorithm~\ref{#1}}
\newcommand{\appref}[1]{Appendix~\ref{#1}}
\newcommand{\tabref}[1]{Table~\ref{#1}}

\title{A global algorithm for the computation of traveling dissipative solitons}

\author{Y.S. Choi \and J. M. Connors}
\institute{Y.S. Choi \and J. M. Connors \at 
University of Connecticut, 
Department of Mathematics, 
341 Mansfield Road U-1009, 
Storrs, CT 06269-3009 \\\email{jeffrey.connors@uconn.edu}}

\spnewtheorem{alg}{Algorithm}{\bf}{\rm}

\begin{document}

\maketitle

\begin{abstract}
An algorithm is proposed to calculate traveling dissipative solitons for the 
FitzHugh-Nagumo equations. It is based on the application of the steepest descent method 
to a certain functional.  This approach can be used to find solitons whenever the 
problem has a variational structure.  Since the method seeks the lowest energy 
configuration, it has robust performance qualities.  It is global in nature, so that 
initial guesses for both the pulse profile and the wave speed can be quite different from 
the correct solution.  Also, 
bifurcations have a minimal effect on the performance.  
With an appropriate set of physical parameters in two dimensional domains,
we observe the co-existence of single-soliton and 2-soliton solutions together with
additional unstable traveling pulses.  The algorithm automatically calculates these 
various pulses as the energy minimizers at different wave speeds.  In addition to 
finding individual solutions, this approach could be used to augment or initiate 
continuation algorithms.  
%

\end{abstract}

\date{}



\section{Introduction}\label{sec:introduction}

Patterns occurring in nature fascinate. They are ubiquitous in all kinds of physical, chemical and biological systems.
{Very often, localized structures are observed, like pulses, fronts and spirals.} 
These 
fundamental building blocks then self-organize into patterns as a result of their mutual interaction. This may involve a pattern front moving
across a homogeneous ambient state after some kind of destabilization.  
{In many cases, the ambient state is} 
uniform in all directions.  Traveling pulses are localized structures
 which relax to the same ambient state; they are therefore especially important in the dynamic transition phase 
 seen in many experiments and model simulations. 
 Stable pulses (which are sometimes called {{\em spots}} in multidimensional domains), whether stationary or moving, can be particle-like in many 
 circumstances and are {often} referred to
 as dissipative solitons in {the} physics literature. {For example, see} the books by \cite{Nish2002, AA2005, L2013} and the references therein. In a three-component 
 activator-inhibitor system studied by both physicists and
 mathematicians, it was observed that
fast moving solitons collide to annihilate one another while slow ones bounce off from one another, and an unstable standing pulse can split into two
solitons; see \cite{PBL2005, EMN2006, KM2008, NTYU2007, HDKP2010}. 

Understanding the mechanisms behind such pattern formation has 
been an on-going struggle since Turing's landmark paper on morphogenesis, \cite{Tur1952};
 investigation has intensified in the last three decades in recognition of the importance {observed for various applications}.  
 Advances in mathematical studies lead to a deeper understanding:
such interactions involve a delicate balance between gain and loss, and the subsequent redistribution of energy and ``mass'' in the system; the ``mass'' can be chemical concentration,
light intensity or current density. Many dissipative soliton models, like Ginsburg-Landau and nonlinear Schrodinger equations, possess variational structures.  
Restricting our attention to reaction-diffusion systems,
activator-inhibitor type equations are the natural choices in modeling these phenomena, as they involve gain and loss; 
see \cite{PBL2005, L2013}. The two component 
FitzHugh-Nagumo equations and the three component activator-inhibitor systems serve as primary
models in such investigations; under suitable parametric restrictions the solutions are minimizers of some variational functionals.

{There are many theoretical studies on these activator-inhibitor 
systems that employ various methods of analysis for one or higher dimensional 
domains in different parametric regimes, for example} 
 \cite{CC2012, CC2015, CCH2016, HS2014, Mur2004, DRY2007, RW2003}. 
We are particularly interested {in the case} when the activator 
diffusivity is small (compared to that of the inhibitor), leading to an 
activator profile with a steep slope.  By using the tool of $\Gamma$-convergence {on the FitzHugh-Nagumo equations} to study its limiting geometric variational problems with a nonlocal term,
existence and local stability of radially symmetric standing spots in ${\mathbb R}^n$ have been completely classified for all parameters \cite{CCR2018, CCHR2018};
they are the first results on the exact multiplicity of solutions to these limiting equations ({from} $0$ up to $3$ standing pulses).
With similar restrictions on parameters in order to employ $\Gamma$-convergence,
a unique traveling pulse solution in the 1D case has also been shown recently \cite{CCF}. However, when we relax the conditions on the parameters
(so that one cannot employ $\Gamma$-convergence analysis), there can be co-existence of a  traveling pulse and two distinct fronts moving in
opposite directions \cite{CC_multi}.

Computational studies are also important, but there can be difficulties.  
Continuation algorithms are a common means to compute traveling and stationary waves. While there is no difficulty in finding 1D traveling spots of the 
FitzHugh-Nagumo equations, the 2D case is different. 
We extract an exact quote from \cite[p.8]{L2013}: ``Only in extreme parameter regimes, where numerical solutions are difficult to obtain, it can be analytically shown that propagating dissipative solitons exist as stable solutions of two-component, two- dimensional, reaction-diffusion systems''; the algorithms 
experience a hindrance.  At the same time, from  \cite[bottom of p.32]{EMN2006}:
  ``it is generally believed that 
{traveling} spots  for a two component system in the
whole  $\mathbb{R}^2$ do not exist''. The difficulty comes as a result of multiple bifurcations in the 2D domains. 
Sometimes, just in some small range of parameters there are 
multiple bifurcations occurring (the cusp in Figure~\ref{fig:fig2d_3} is related to bifurcation), resulting in many close-by solutions; in such situations it is not easy to find a good enough
 initial guess for the continuation algorithm.  Even when sucessful, convergence is not guaranteed.

{Another computational approach is to 
feed reasonable initial profiles into the time dependent problem and perform numerical 
time stepping; one then hopes that
it results in a more or less steady profile moving at a uniform speed after a long run. 
All unstable waves can never be found using such a method.
Even if a stable traveling wave exists, it is easy to miss the solution since usually they 
 are not global attractors for reaction-diffusion systems, as in the case of the~\fhn equations.  
 This method sometimes provides a good initial guess for a Newton-type algorithm or a numerical 
 shooting method if we are interested in 
 very accurate solutions.  
As we are interested in cases when $d$ is small, the resulting multiple temporal and spatial scales  in 
\eqref{eqn:fhn1} will  induce
excessive  computational effort, unless the initial approximation is extremely good so that the convergence 
to the traveling wave takes place in a short time.} 

{If we} restrict ourselves to dissipative solitons for which variational formulations are common, we {may} exploit the fact that 
they are minimizers.  
{In this paper, we develop a robust, global steepest descent 
algorithm to find traveling pulse solutions of the FitzHugh-Nagumo equations 
both in 1D and on an infinite strip domain in 2D with zero Dirichlet boundary 
conditions.  It is based on some recent theoretical understanding; see \cite{CC2012, CC2015}.}  
The algorithm works even without a good initial guess.  If a 
bifurcation occurs, it simply tracks the lowest energy configuration and filters out other high energy solutions generated as a result of bifurcation.  
Multiple bifurcations in a vicinity will therefore not affect the algorithmic 
performance. That explains why we are able
to find quite a few stable and unstable spots with relative ease in our computations. Future robustness tests on this algorithm will be conducted with other boundary conditions, widening the width of the strip domain,  and for traveling fronts.  
Another study \cite{CCD_new} also shows that the algorithm can easily compute all three traveling waves proved in \cite{CC_multi};
in fact, in addition to these 3 stable minimizers there are 2 unstable waves with the same physical parameters. 
The algorithm can also be {applied toward} 
the 3-component activator-inhibitor systems, which also possess variational 
structures.  

We consider the~\fhn equations with domain $\Omega\subset \mathbb{R}^n$ in space, for 
time $0\leq t <\infty$.  Specifically, $\Omega=\mathbb{R}$ if $n=1$ and 
$\Omega=\mathbb{R}\times (-L,L)$ if $n=2$, with $0<L<\infty$.  The equations are: 
\begin{equation}
 \begin{array}{rl}
\ds u_t &= \ds \Delta u +\frac{1}{d}\left( f(u)-v\right) , \\
\ds v_t  & = \ds  \Delta v +u-\gamma \, v , 
\end{array} 
\label{eqn:fhn1} 
\end{equation}
plus initial conditions. 
Boundary conditions are needed for $n=2$; we consider $u=v=0$ on the boundary.  
Here $d$ and $\gamma$ are positive constants, and
$f(u)\equiv u(u-\beta )(1-u)$ with $0<\beta <1/2$ being a fixed constant.  


Our algorithm is global in the sense that it does not require 
a good initial guess of the solution, only some guess in the set $\A$ that we show later 
is very easy to construct.  The wave speed is found as a root of a certain 
functional.  Multiple roots correspond to distinct waves, {hence multiple solutions for the same values of $d$, $\gamma$ and $\beta$.}   
We develop this approach here for the specific domains $\Omega$ described 
above, but it can be easily extended to more general boundary conditions and 
domains via standard adaptations of the variational techniques.  
After deriving our method, we will 
computationally illustrate how it is used to determine the existence and 
multiplicity of traveling pulse solutions to (\ref{eqn:fhn1}), and 
simultaneously compute the corresponding wave speeds and pulse profiles.  
In 2D, we observe a bifurcation and point out how the energy minimization 
property allows the algorithm to adapt automatically and find the dissipative soliton.

\subsection{An illustration of robustness}\label{sec:robustExample}  
As an example of the global convergence with $n=1$, choosing $\beta =1/4$, $d=0.0005$ and 
$\gamma=1/16$, our proposed method 
is able to identify one stable and one unstable wave profile, corresponding 
to different wave speeds $c_0$ and $c_1$, respectively.  
In order to check the stability of 
our computed traveling pulses, they are input as initial data into a parabolic 
solver for the time-dependent equation~\eqref{eqn:fhn1}. While some break up 
quickly, which indicates an unstable traveling wave; others just translate with 
the computed speed from our descent algorithm, verifying that they are stable. 

The situation is illustrated in~\figref{fig:fig0}; 
on the {top}, we show the result of inserting the unstable traveling 
wave profile into a parabolic solver: the wave just decays rapidly in time.  This demonstates the need of a 
global algorithm; if the initial guess for a stable wave profile is not close 
enough, the parabolic solver may not find it.  In contrast, our 
proposed method can start with the same unstable profile and 
be used to 
find the stable profile with speed $c_0$ as well. 
In the center 
of~\figref{fig:fig0}, this process is shown at various iteration counts, $n$, 
of the algorithm.  On the {bottom} of~\figref{fig:fig0}, we show the 
result of inserting the computed, stable wave profile into the parabolic 
solver. 
\begin{figure}[h!]
\includegraphics[width=\textwidth]{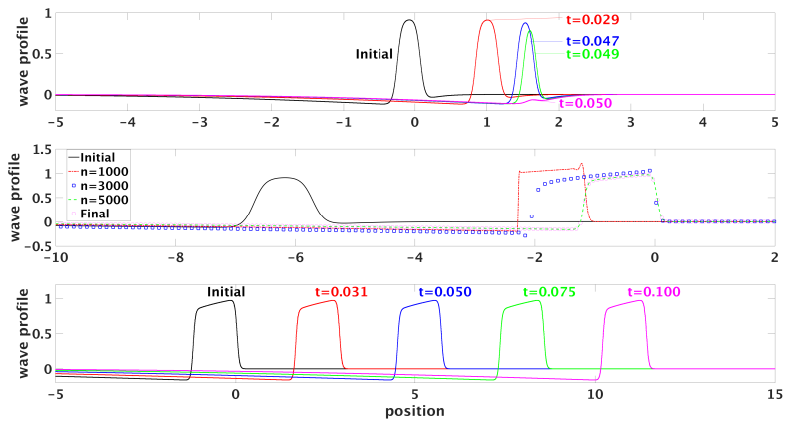}
\caption{{Top}: an initial guess for a traveling wave profile that decays rapidly in a parabolic solver.  Center: the same initial guess is used to initiate the proposed global method (after a change of variables, described later).  Snapshots are shown for various iteration counts, $n$.  {Bottom}: the resulting profile now remains steady in a parabolic solver.}\label{fig:fig0}
\end{figure}

The remainder of this paper is organized as follows.  
In~\secref{sec:vf} we give a variational formulation of the~\fhn equations so that a critical point corresponds to a traveling pulse, 
provided its wave speed also satisfies an auxiliary scalar algebraic equation.
The admissible set $\A$ in which we look for the critical point has to be described carefully.
In~\secref{sec:steepestDescent} we present a steepest descent algorithm which, together with the 
auxiliary scalar algebraic equation, computes a traveling pulse profile as well as its wave speed.
The details in~\secref{sec:vf}-\secref{sec:steepestDescent} are presented for 
one dimension of space.  In~\secref{sec:alg2D}, we explain the extension to two 
dimensions, which requires fairly minor modifications that are standard for 
variational analyses.  
In~\secref{sec:results_1d} we present, compare and perform some checks on numerical results from our algorithm in one dimension of space.  
The fastest traveling pulse speed has an asymptotic limit as $d \to 0$, see  \cite{CC2015}; we check our numerical wave speed against this 
theoretical result and find excellent agreement. 
In~\secref{sec:results_2d} we present numerical results for two-dimensional 
domains and find as many as four traveling pulse solutions (with different wave speeds) 
for a single choice of 
parameters.  
A bifurcation separates these pulses into two distinct groups qualitatively, but the 
algorithm automatically computes them all without any special or prior knowledge of the pulse profiles. 
In~\secref{sec:summary} we give a short summary of {our results} 
and {discuss} some future work.  In the {remainder} of {this} paper, the 
terminologies {`pulse', `spot' and `dissipative soliton'} essentially mean 
the same thing.  We use `pulse' in 1D to conform to mathematicians' preference. In 2D, if 
the pulse is stable, we call it a (dissipative) soliton to emphasize its localized 
structure with {a} particle-like property.  

\section{A variational formulation for traveling pulse} \label{sec:vf}
In this section we take $n=1$ for~\eqref{eqn:fhn1}; the case $n=2$ is handled 
in Section~\ref{sec:alg2D}.  Following \cite[Theorems 1.1]{CC2015}
we impose the restriction $0<\gamma <4/(1-\beta)^2$. 
This is a necessary and sufficient condition for the straight line $u=\gamma v$ to cut the curve $v=f(u)$ only at the origin in the 
$(u,v)$ plane.
This guarantees that $(u,v)=(0,0)$ is the only constant-state equilibrium 
solution and hence eliminates the possibility of a traveling front.  
Under such a condition the cited Theorem 1.1 ensures that a traveling pulse solution exists when $d$ is sufficiently small with fixed
 $\gamma$ and $\beta$. 

We look for traveling pulse solutions $(u(x,t),v(x,t))\in \mathbb{R}^2$ of~\eqref{eqn:fhn1}
with $-\infty <x<\infty$ and $t\geq 0$.  For a wave speed $c$, which is not yet known, we require that 
\begin{equation}
u(x,t) = \tilde{u} (c(x-ct)) \ \ \text{and} \ \ v(x,t) = \tilde{v} (c(x-ct)) 
\label{eqn:fhn2} 
\end{equation}
for some smooth functions $\tilde{u}: {\mathbb R} \to {\mathbb R}$ and  $\tilde{v}: {\mathbb R} \to {\mathbb R}$.
Dropping the tilde in the notation and use $x$ to denote $\xi= c(x-ct)$, 
the traveling pulse problem is to find $(u,v,c)$ satisfying 
\begin{equation}
 \begin{array}{rl}
\ds dc^2 \partialxx{u} +dc^2 \partialx{u} +f(u)-v &=0  , \\
\ds c^2 \partialxx{v} +c^2 \partialx{v}+u-\gamma \, v  &=0 
	\end{array}
	\label{eqn:fhn3} 
\end{equation}
with $(u,v)\to (0,0)$ as $|x|\to \infty$.  We can always let $c>0$.  

We introduce Hilbert spaces $\Lex({\mathbb R})$ and $\Hex({\mathbb R})$, corresponding to the 
inner products
\begin{equation}
\begin{aligned}
\langle v,w\rangle_{\Lex} &\equiv \int_{\mathbb{R}} e^x  vw \, dx \\
\text{and} \ \ \langle v,w\rangle_{\Hex} &\equiv \int_{\mathbb{R}} e^x \left\{ \partialx{v}\partialx{w} +vw  \right\}\, dx , 
\end{aligned}
\label{eqn:innerProd}
\end{equation}
respectively.  The induced norms are denoted by $\| \cdot \|_{\Lex}$ and 
$\| \cdot \|_{\Hex}$.  
A variational approach will be used to find weak solutions 
$(u,v)\in (\Hex)^2$ to ~\eqref{eqn:fhn3}. 

By solving (\ref{eqn:fhn3}b)
 we write $v=\Lc u$, where 
$\Lc : \Lex\to\Lex$ is a linear operator.  
It can be verified 
that $\Lc$ is a self-adjoint operator 
on $\Lex({\mathbb R})$, i.e. $\langle u_1, \Lc u_2 \rangle_{\Lex} = \langle \Lc u_1, u_2 \rangle_{\Lex}$ for any
$u_1, u_2 \in \Lex({\mathbb R})$. A way to see this is to write (\ref{eqn:fhn3}b) as $c^2 (e^x v_i')' - \gamma e^x v_i= - e^x u_i$
for $i=1,2$; then it is easy to check $\langle u_1, v_2 \rangle_{\Lex}=\langle v_1, u_2 \rangle_{\Lex}$, provided
there is sufficient control on $u_i$ and $v_i$ so that the boundary terms at infinity arising from integration by parts can be discarded.

For any given $c>0$,
consider the functional $J_c : \Hex \to \mathbb{R}$ defined by 
\begin{equation}
J_c (w) \equiv 
\int_{\mathbb{R}} e^x \left\{ \frac{dc^2}{2} w'^{\,2} 
+ \frac{1}{2} w\, \Lc w +\F (w) \right\}\, dx, 
\label{eqn:Jc} 
\end{equation}
where 
\begin{equation} \label{eqn:F}
\F (\xi) = -\int_0^\xi f(\tau) \, d\tau = \frac{\xi^4}{4} -\frac{(1+\beta) \xi^3}{3} 
	+\beta \frac{\xi^2}{2} .
\end{equation}
Making use of the self-adjointness of $\Lc$ on $\Lex({\mathbb R})$, its Fr\'echet derivative  is given by
\begin{equation}
J_c' (w)\phi \equiv \int_{\mathbb{R}} e^x \left\{
 dc^2 w' \,\phi' +\Lc w\, \phi -f(w)\phi \right\} \, dx 
\ \  \ \text{for all} \ w, \phi \in \Hex({\mathbb R}) .  
\label{eqn:JcPrime}
\end{equation}
A critical point $u$ of $J_c$ in $\Hex({\mathbb R})$ will satisfy the Euler-Lagrange equation associated with $J_c$, namely
\begin{equation} \label{eqn:intDiff}
- dc^2 (e^x u')'  - e^x f (u) + e^x \Lc u=0 \;.
\end{equation}
This integral-differential equation is equivalent to the Fitzhugh-Nagumo equations~\eqref{eqn:fhn3}.
Recall that  $\Hex({\mathbb R}) \subset C({\mathbb R})$, we are now ready for the following. 
\begin{defn}  A function $w\in C(\mathbb{R})$ is in the class $\MPM$ if 
there exist $-\infty\leq x_1 \leq x_2 \leq \infty$ such that (a) 
$w(x)\leq 0$ for all $x\in (-\infty,x_1] \cup [x_2,\infty)$ and (b) 
$w(x)\geq 0$ for all $x\in [x_1,x_2]$. 
\label{def:MPM}
\end{defn}
\begin{remark} \label{remark1}
(a) In the above definition, the choice of $x_1,x_2$ is not necessarily unique.
If $x_1=-\infty$ and $x_2=\infty$, then $w \geq 0$ on the real line. 
In case $x_1=x_2=\infty$, then $w \leq 0$ on the real line. Both examples 
are included in the
class $-/+/-$. \\
(b) A function $w$ is said to change sign twice, if 
$w \leq 0$ on $(-\infty,x_1] \cup [x_2, \infty)$, $w \geq 0$ on $(x_1,x_2)$ and $w\not \equiv 0$ in each of such
three intervals. 
\end{remark}

As $\beta<1/2$, there is a unique $\beta_1$ such that
$\beta<\beta_1<1$ with
$\F (\beta_1)=0$. 
 In addition, we take a constant $M_1 = M_1 (\gamma) \geq 1$ such that 
$f(\xi ) \geq  1/\gamma$ for all $\xi \leq -M_1$.  A class $\A$ of 
admissible functions to be employed in the variational argument is defined as 
follows.
\begin{defn}
\begin{equation} \label{eqn:A}
\A \equiv \left\{ w\in \Hex :   \| w\|_{\Hex}^2 = 2, \, 
-M_1 \leq w \leq 1,  w \; \mbox{is in the class}\; \MPM\right\} .
\end{equation}
\end{defn}
We restrict attention to $J_c: \A \to {\mathbb R}$. 
A global minimizer is known to exist in $\A$ for any fixed $c>0$. We can therefore let
\begin{equation} \label{eqn:jc1}
	\Jc \equiv \min_{w\in \A} J_c (w) , \ \ \text{for all} \ c>0.
\end{equation}
When $d \leq d_0$ for some sufficiently small $d_0$, it can be shown that 
$\Jc<0$ when $c$ is small, $\Jc>0$ when $c$ is large, and $\cal J$ is a continuous function. 
By the intermediate value theorem there is at least one $c_0>0$ such that 
$\mathcal{J} (c_0) =0$. Suppose $u_0$ is a
global minimizer in $\A$ when $c=c_0$, and we let $v_0={\cal L}_{c_0} u_0$, then $(u_0,v_0)$
can be shown to be smooth and $(u_0,v_0,c_0)$ is
a traveling pulse solution; thus $u_0$ is an unconstrained critical point of $J_{c_0}$.
We will give a heuristic argument 
in subsection~\ref{sec:zeroJc}  
on why $\Jc=0$ determines the wave speed. The rigorous proof is in \cite{CC2015}.

For any fixed $c$
we will construct a steepest descent algorithm in the next section to find  $\Jc$ and the corresponding minimizer $u$ of $J_c$.
A spatial translation of any traveling pulse solution 
remains a traveling pulse. This continuum of solutions will induce both theoretical and numerical difficulties. 
The condition $\|u_0 \|_{\Hex}^2=2$ in $\A$ makes sure that $u_0(\cdot+a)$ is not in $\A$
for any non-zero $a \in {\mathbb R}$ 
and hence eliminates translation of solution.

\begin{remark} \label{remark2}
The admissible set $\A$ defined above differs slightly from that in (2.9) of \cite{CC2015}, because 
\begin{enumerate}
\item the weighted $H^1$ norm  $\| w \|_{\Hex}$ employed here  is equivalent to
the norm  $\sqrt{\int_{\mathbb R} e^x w_x^2 \,dx}$ used in \cite{CC2015}. The new norm may be better 
when we perform numerical computations after truncating the real line to a finite domain.
\item Subsequent analysis in \cite{CC2015} leads to tighter bounds on a local minimizer of $J_c$;
it allows us to use the simpler
admissible set \eqref{eqn:A} for computational purposes.
\end{enumerate}
\end{remark}

Suppose there are multiple traveling pulse solutions in  ${\cal A}$ for the same physical parameters. If $c_0$ is
 the fastest wave speed among these solutions,
it follows from Theorem 1.3 in \cite{CC2015} that 
\begin{equation} \label{eqn:fastSpeed}
d c_0^2 \to \frac{(1-2\beta)^2}{2} \ \ \text{for the fastest wave as} \ d \to 0.
\end{equation}
Our proposed numerical algorithm will compute all the traveling pulses irrespective of whether they are fast 
or slow waves; indeed we do find multiple waves in some physical parameter regime.
We will employ~\eqref{eqn:fastSpeed} to check the accuracy for our algorithm.

To investigate the stability of our computed traveling pulses, we feed them (after rescaling to the original variables)
  as initial conditions 
 into the time-dependent equations. Both stable and unstable traveling pulses are found
 in the admissible set $\A$. The parabolic
 solver serves as an independent check for our algorithm.

\subsection{Why the auxiliary equation ${\cal J}(c_0)=0$ determines the wave speed}\label{sec:zeroJc} 

It is not immediately clear that when ${\cal J}(c_0)=0$ then $c_0$ is the traveling pulse speed.
We give a heuristic argument in this subsection to enhance the understanding of our algorithm.

Let $c>0$ and $u \in \A$ be a minimizer of $J_c$.
Suppose  
\begin{enumerate}
\item the inequality constraints on $u$ is inactive, i.e. $-M_1<u<1$ on the real line;
\item the oscillation requirement $u \in \MPM$ is inactive, i.e. there is no interval on which 
$u=0$. This leads to a smooth $u$;
\item $u$ and its dervative have fast decay as $x \to \infty$ and they remain bounded
as $x \to -\infty$.
\end{enumerate}
We introduce a Lagrange multiplier $\lambda$ to remove the last remaining equality constraint $\| u \|_{\Hex}^2=2$
in $\A$, therefore 
$u$ is an unconstrained critical point of ${\cal I}_c$ with
\begin{equation} \label{eqn:Ic}
	{\cal I}_c(w)=J_c(w) + \lambda \left( \int_{\mathbb R} e^x \frac{1}{2} \left( w'^{\,2} + w^{2}\right) \,dx  -1 \right) \;.
\end{equation}
Hence for all $\phi \in \Hex$
\begin{equation} \label{eqn:icPrime}
0= {\cal I}'_c(u)\phi=J'_c(u) \phi + \lambda \int_{\mathbb R} e^x \left( u' \, \phi' + u\phi \right) \, dx\;.
\end{equation}
Set $\phi=u'$. Using \eqref{eqn:JcPrime} the above equation can be reduced to
\begin{equation*}
\int_{\mathbb R}  e^x \, \left\{ \frac{\partial}{\partial x} \left( \frac{dc^2}{2} u'^{\,2} + \F(u)+ \frac{1}{2} u \, \Lc u 
\right)  + \frac{\lambda}{2} \frac{\partial}{\partial x} (u'^2 +u^2) \right\} \, dx =0 
\end{equation*}
by using the self-adjointness of $\Lc$ on $\Lex({\mathbb R})$. 
An integration by parts  leads to $J_c(u)+ \lambda=0$ due to the assumed asymptotic behavior of $u$ and its derivative 
for large $|x|$.
Suppose $c_0$ satisfies $J_{c_0}(u)=0$, then $\lambda=0$.  Write this $u$ as $u_0$.

We now have  $u_0$ being a minimizer of ${\cal I}_{c_0}$ and $v_0={\cal L}_{c_0} u_0$.
As $u_0$ is an unconstrained critical point of ${\cal I}_{c_0}$,
 it will satisfy the Euler-Lagrange equation associated with ${\cal I}_{c_0}$. With $\lambda=0$ in \eqref{eqn:icPrime}, this
Euler-Lagrange equation is the same as $J'_{c_0}(u_0)=0$, which simplifies to~\eqref{eqn:intDiff}.
 In other words $(u_0,v_0,c_0)$ satisfies the FitzHugh-Nagumo
equations~\eqref{eqn:fhn3}.

\subsection{Minimizer $u$ is positive somewhere} \label{sec:signs}

\begin{lem} \label{lem:positive}
Let $c>0$ and $u$ be a global minimizer of $J_c: \A \to {\mathbb R}$. Then $\max u>0$.
\end{lem}
\begin{proof}
Suppose $u \leq 0$ for all $x$. As $u \not \equiv 0$, we have $\min u <0$. Define $\tilde{u} \equiv -u$. 
Then $\tilde{u} \in \A$ is non-negative and
\begin{eqnarray*}
J_c(\tilde{u})&=&  
\int_{\mathbb{R}} e^x \left\{ \frac{dc^2}{2} \tilde{u}'^{\,2} 
+ \frac{1}{2} \tilde{u}\, \Lc \tilde{u} +\F (\tilde{u}) \right\}\, dx \\
&<&  \int_{\mathbb{R}} e^x \left\{ \frac{dc^2}{2} {u}'^{\,2} 
+ \frac{1}{2} {u}\, \Lc {u} +\F ({u}) \right\}\, dx \\
&=& J_c(u) ,
\end{eqnarray*}
because from \eqref{eqn:F} we have $\F(\xi)> \F(-\xi)$ if {$\xi < 0$}, 
while the gradient energy {and nonlocal energy terms remain the same.} This contradicts $u$ being a global minimizer in $\A$.
\hfill  $\Box$
\end{proof}

Starting with an initial guess $w^0$, as the successive iterates $w^{(n)}$ from the steepest descent algorithm,
to be proposed in the next section,
 get closer to the minimizer $u$, the above Lemma ensures that
 $\max w^{(n)}>0$ for large enough $n$.
If $c=c_0$ with ${\cal J}(c_0)=0$, we have a stronger result:  $\max u_0 \to 1$ as $d \to 0$
 \cite[Theorem 8.6]{CC2015}, and $u_0$ changes signs exactly twice.  
 
\subsection{Monotonicity of $\Jc$ with respect to $d$} \label{sec:monotone}
 
 This is a simple observation, but will serve as a useful  guide to choose a proper range of $d$ for traveling pulse  
 phenomena in our numerics. It is clear that  $J_c$  depends on the parameter $d$ 
 besides $c$. Fix any $c$, $\gamma$, $\beta$ and $w \in \A$. Suppose $d_2 \geq d_1$,
 then $J_c(w;d_2)-J_c(w;d_1)= \int_{\mathbb R} e^x   \frac{{(d_2-d_1)}c^2}{2} w'^2 \, dx >0$ so that
${\cal J}(c;d_2) \geq {\cal J}(c;d_1)$. 
 
 Now let $w$  be a minimizer of $J_c(\cdot \,; d_1)$ when $d=d_1$. It follows that
 \begin{eqnarray*}
 J_c(w;d_2) & \leq &{\cal J}_c(c;d_1)+ \frac{(d_2-d_1)c^2}{2} \int_{\mathbb R} e^x (w'^2+w^2) \, dx \\ 
 & = & {\cal J}_c(c;d_1)+ (d_2-d_1)c^2.
 \end{eqnarray*}
Hence 
\begin{equation} \label{eq:orderJ}
{\cal J}(c;d_1) \leq {\cal J}(c;d_2) \leq {\cal J}(c;d_1) + (d_2-d_1) c^2 \;.
\end{equation}

\section{A steepest-descent method for computing $\Jc$}\label{sec:steepestDescent} 
We continue in this section with $n=1$; one dimension in space 
for~\eqref{eqn:fhn1}, with the case $n=2$ discussed in Section~\ref{sec:alg2D}.  
Given a $c>0$,
we would like to find a global minimizer $u$ of $J_c $ in the admissible set  $\A$ so that we can  compute
$\Jc=J_c(u)$.
Qualitative features of $u$ have been given in 
\cite[Theorem 1.2]{CC2015} when $d$ is small;  however  quantitatively this
is only a rough guess of what the minimizer profile would be like and a global
algorithm is warranted.
At any  $w\in \A$,
since we seek a global minimizer,
it is natural to use the steepest descent tangent vector direction (with $\Hex$ norm as the metric)
at $w$ on the manifold
${\cal M} \equiv \{p \in \Hex: \|p \|^2_{\Hex}=2 \}$. 
Following the steepest descent direction will eventually lead us to the minimizer $u$.

If we do not need to stay on the manifold, for any small arbitrary change $\epsilon \phi$, we have
$J_c(w+\epsilon \phi)-J_c(w) = \epsilon J_c'(w) \phi + O(\epsilon^2)$. Hence the steepest descent direction is to optimize 
$J'_c(w) \phi$ subject to unit norm on $\phi$. In our case
a modification  is necessary to stay on the manifold ${\cal M}$.

Let the steepest descent direction in our case at any given $c>0$ and $w \in \A$ be denoted 
by $q=q(w,c)$, which is normalized so that $\| q\|_{\Hex}^2 =2$. 
If $\epsilon$ is small, we want $\tilde{w}=w+\epsilon q$ to satisfy
$\| \tilde{w} \|_{\Hex}^2=2$ to leading order of $\epsilon$. This amounts to
enforcing the orthogonality condition $\langle w,q \rangle_{\Hex} =0$ so that $q$ is a tangent vector 
on the manifold ${\cal M}$.
A small (second order) correction on $\tilde{w}$, to be described later, will give
a new $w_{new}$ on the manifold ${\cal M}$ and  result in  $J_c(w_{new}) < J_c(w)$.

Following the idea advocated in \cite{CM1993},
we introduce Lagrange multipliers $\lambda$ and $\mu$ 
to remove the equality constraints $\| q \|^2_{\Hex}=2$ and $\langle w, q \rangle_{\Hex}=0$. Therefore
$q$ can be found as an unconstrained critical point of
\begin{equation}
\begin{aligned} 
K_c \left(\phi \right) &\equiv J_c' (w)\phi
+\lambda \left(\frac{1}{2} \| \phi \|_{\Hex}^2 -1\right) 
+\mu \langle w,\phi \rangle_{\Hex}
\ \ \text{for all} \ \phi \in \Hex .  
\end{aligned} 
\label{eqn:K}
\end{equation}
Hence we have
\begin{equation} \label{eqn:kcPrime}
K'_c(q) =0 
\end{equation}
with
\begin{equation} 
K'_c(\phi )p  = J_c' (w)p 
+\lambda \langle \phi,p \rangle_{\Hex} +\mu \langle w,p \rangle_{\Hex}  \ \text{for all} \ p\in \Hex.  
\label{eqn:qWeak1}
\end{equation}
Combining \eqref{eqn:kcPrime} and \eqref{eqn:qWeak1}, we arrive at
\begin{equation} \label{eqn:qWeak}
 J_c' (w)p 
+\lambda \langle q,p \rangle_{\Hex} +\mu \langle w,p \rangle_{\Hex}=0  \ \text{for all} \ p\in \Hex.  
\end{equation}
Upon inserting $p=w$ in~\eqref{eqn:qWeak}, 
\begin{equation*}
J_c' (w)w 
+\lambda \langle q,w\rangle_{\Hex} +\mu \|w\|^2_{\Hex} = 0.  
\end{equation*}
As $\|w\|^2_{\Hex}=2$ and $\langle q,w \rangle_{\Hex}=0$, it is immediate that
\begin{equation}
\mu =-\frac{1}{2} J_c' (w)w ,
\label{eqn:muValue}
\end{equation}
which can be calculated, as $w \in \A$ is given.

Now we are ready to solve the linear equation~\eqref{eqn:qWeak} for $\lambda q$, which is parallel to the seach direction.  
It is not necessary to calculate $\lambda$.  Indeed by choosing $p=\lambda q$ in~\eqref{eqn:qWeak},
$J_c'(w)(\lambda q) = -\lambda^2 \| q\|_{\Hex}^2 <0$.  Thus we have 
\begin{equation} \label{eqn:sign}
J_c (w+\epsilon \lambda q)< J_c (w) \ \ \text{for small enough values} \ \epsilon>0 \;.
\end{equation}  

We rewrite ~\eqref{eqn:qWeak}  in strong form as follows
\begin{equation*}
 - \frac{\partial}{\partial x}\left( e^x \frac{\partial}{\partial x}(\lambda q + \mu w) \right)
 + e^x (\lambda q + \mu w) - \frac{\partial}{\partial x} \left(  dc^2  e^x \partialx{w} \right) - e^x f(w) + e^x \Lc w=0 
 \end{equation*}
with $\lambda q$ being the only unknown in this equation and 
 $|q(x)|\to 0$ as $|x|\to \infty$.  
In order to reduce numerical errors  in computing
$\lambda q$ later on, we 
introduce the auxiliary function $w^* = \lambda q +(\mu+dc^2) w$, which then 
solves 
\begin{equation}
-\Delta {w^*}-\partialx{w^*} +w^* = dc^2 w-\Lc (w)+f(w)
\label{eqn:wStar}
\end{equation}
with $|w^* (x)|\to 0$ as $|x|\to \infty$. One can write down its Green's function and a unique solution $w^*$ exists.
The map $\Q :\Hex \to \Hex$ such that $w\to \Q (w)\equiv w^*$ is therefore
well-defined.  

Nonlinear equations can only be solved numerically by an iterative scheme. Instead of writing $w^{(n)}$ to denote
the $n^{th}$ iterate, we will henceforth employ $w^n$ instead for notation simplicity. It should be clear from the context that 
we do not mean the $n^{th}$ power of $w$. Similarly $\alpha^n$ will be used for descent step size instead of $\alpha^{(n)}$.

Let $c>0$ be fixed.  Given an approximation $w^n\in \Hex$ for the minimizer 
of $J_c$, we solve $\Q (w^n)$ numerically and update by means of
\begin{equation}
w^{n+1} =w^n+\alpha^n \left( \Q(w^n)-(\mu+dc^2) w^n \right).  
\label{eqn:update1}
\end{equation}
Observe that $\Q(w^n)-(\mu+dc^2) w^n =\lambda^n q^n$, and $0<\alpha^n<1$ is some descent step 
size, to be discussed later. The positivity of $\alpha^n$ is a consequence of~\eqref{eqn:sign}.
Two problems need to be addressed.  One is that $w^{n+1}$ need not be in the oscillation class $\MPM$.
  The other problem is that even if $\| w^n\|_{\Hex}^2 =2$, the 
constraint $\| w^{n+1}\|_{\Hex}^2 =2$ need not be satisfied.  In either case, 
$w^{n+1}$ need not be in the class $\A$.  We introduce two additional operators 
to address these issues. 

The first operator $r$ clips any portions of the wave profile that may become 
positive outside of the region $[x_1,x_2]$ in~\defref{def:MPM}.  The 
clipped profile will then develop kinks and becomes non-smooth. 
\begin{defn}\label{defn:clipDomain}
\[
{C}_0^+ \equiv \left\{ w\in {C}(\mathbb{R} ): \ 
w(x)>0 \ \text{for some} \ x\in \mathbb{R} \ \text{and} \lim_{|x|\to \infty} w(x) =0 \right\} .
\]
\end{defn}
\begin{defn}\label{defn:clip}
Let $w\in {C}_0^+$ be given and define 
\[
\overline{x} \equiv \max \left\{ x\in\mathbb{R}: \ 
w(x) =\max_{y\in\mathbb{R}} w(y) \right\} . 
\]
Let $(x_1,x_2)$ be the largest open interval containing $\overline{x}$ such 
that $w(x)>0$ for all $x\in (x_1,x_2)$.  We define the clipping 
operator $r: {C}_0^+ \to \MPM$ such that for any $w\in {C}_0^+$,
\begin{equation}
r(w)(x) = \left\{\begin{array}{cc}
w(x), & \ \text{if} \ x\in (x_1,x_2) \ \text{or} \ w(x)\leq 0, \\
0,& \ \text{otherwise} . 
\end{array}\right.
\end{equation}
In case $w \leq 0$ everywhere, we define $r(w)=w$. 
\end{defn}
The second operator is a pure translation  of a given profile along the x-axis. Such
a shift operator is all that is required to enforce the constraint 
$\| w^{n+1} \|_{\Hex}^2 =2$.  
\begin{defn}\label{defn:shift}
The shift operator $s:\Hex \to \Hex$ is defined such that for any $w\in \Hex$,
\begin{equation}
\begin{aligned}
s(w) &= w(\cdot-\log \frac{1}{\omega} ), 
\end{aligned}
\end{equation}
 where $ \omega = \frac{1}{2} \|w\|^2_{\Hex}$.
\end{defn}
\begin{lem} \label{lem:shift}
	Given any $w\in \Hex$, $\| s(w)\|_{\Hex}^2 =2$.  
\end{lem}
\begin{proof}
Let $\omega=\frac{1}{2} \|w\|^2_{\Hex}$  and
$a=\log (1/\omega)$.  It follows from~\defref{defn:shift} that 
\begin{equation*}
\begin{aligned}
\| s(w)\|^2_{\Hex} 
&=\int_{\mathbb{R}} e^x \left\{ \left(w'(x-a) \right)^2 +\left(w(x-a)\right)^2 \right\}\, dx .
\end{aligned}
\end{equation*}
Via the change of variables $y=x-a$, 
\begin{equation*}
\begin{aligned}
\| s(w)\|^2_{\Hex} 
   & =e^{a} \| w\|_{\Hex}^2 = 2.  
\end{aligned}
\end{equation*}
\end{proof}
Given $c>0$,
our steepest descent algorithm to compute $\Jc$  is as follows. 
\begin{alg}\label{alg:alg1}
Choose fixed parameters $0<\theta <1$ (see below), a relative error 
tolerance $0<\delta_1 <1$ and absolute error tolerances $0<\delta_2 <1$ 
{and 
$0<\delta_3 <1$}.  
Given an initial guess $w^0\in \A$ with $\max w^0>0$
 and initial descent step size 
$0<\alpha^0 <1$, we iterate as follows to generate updates $w^n$ for 
$n=1,2,\ldots$. 
\begin{enumerate}
\item Compute $v^n= \Lc w^n$.
\item Set $\mu^n =-\frac{1}{2} J_c'(w^n)w^n$.  
\item Set $Q^n  =\Q(w^n)$.  
\item Set $\tilde{w}^{n+1} = w^n +\alpha^n \left( Q^n-(dc^2+\mu^n) w^n \right)$.
\item Set $w^{n+1}(x) = s\left(r\left(\tilde{w}^{n+1}\right)\right)$.
\item Check descent; if $J_c (w^{n+1})>J_c (w^{n})$ then replace 
	$\alpha^n \leftarrow \theta \alpha^n$ and go to (3).  
\item Update the step size, $\alpha^{n+1}$. 
\item Repeat (1)-(6) until 
\begin{equation}
\left. 
\begin{array}{c}
J_c (w^n)\leq J_c (w^{n+1})+\max \left\{ \delta_1 |J_c (w^{n+1})|, \delta_2 \right\} \\ 
\text{and} \qquad 
\sup_{x\in \Omega} \left| w^{n+1} (x) - w^{n} (x) \right| \leq \delta_3 
\end{array}\right\} . 
\label{eqn:stop1}
\end{equation}
\end{enumerate}
\end{alg}
Heuristic methods for step (6) are discussed later.  

The first part of the stopping criterion~\eqref{eqn:stop1} is implemented with 
$\delta_2 \ll \delta_1$.  This ensures that the relative error is small 
when $|J_c (w^{n+1})|> \delta_2/\delta_1$.  In practice, when 
$|J_c (w^{n+1})|< \delta_2/\delta_1 \ll 1$ is very small we cannot enforce 
the relative accuracy constraint due to round-off error effects.  In such an event, 
 the criterion~\eqref{eqn:stop1} still requires the absolute error to be 
small. We note that this latter case is near the regime $\Jc=0$ that we are most 
interested in. {The role of $\delta_3$ is to ensure that the profile 
is converged completely in the tail region of a pulse (decaying to zero in the 
direction opposite of the wave propagation).  This is necessary because the 
exponential weight in the functional $J_c$ greatly diminishes the effect of tail 
perturbations on the energy.  In other words, an accurate functional value is achieved 
numerically before the tail is fully resolved.  Numerically, the supremum of 
$\left| w^{n+1} (x) - w^{n} (x) \right|$ is interpreted as the maximum over 
all grid points.}

\section{Extension of the method for two dimensions in space}\label{sec:alg2D} 
Here we discuss the adaptation of Algorithm~\ref{alg:alg1} for the case of two 
dimensions in space.  There are still many gaps in the study of traveling waves 
for the~\fhn equations in multiple dimensions.  For algorithmic purposes, the key open issue is 
the definition of the admissible set $\A$.  It is not clear how to define a 
class of functions like~$\MPM$ in multiple dimensions, so the clipping operator 
is not defined in this case, presently.  However, we have found in our 
experiments that the clipping operation was not necessary to compute traveling 
pulse solutions in two dimensions.  It seems to be sufficient to take the step 
size $\alpha^n$ in the algorithm to be small enough.  
{The advantage of the clipping operator in 1D is to allow for larger step sizes.}

We look for traveling wave solutions $(u(x,y,t),v(x,y,t))\in\mathbb{R}^2$ of~\eqref{eqn:fhn1} 
with $(x,y)\in\Omega=(-\infty,\infty)\times(-L ,L)$ and $t>0$.  
Let
\begin{equation}
u(x,y,t) = \tilde{u} (c(x-ct),cy) \ \ \text{and} \ \ v(x,y,t) = \tilde{v} (c(x-ct),cy)   
\label{eqn:fhn2_2d} 
\end{equation} 
for some smooth functions $\tilde{u}: \Omega \to {\mathbb R}$ and  $\tilde{v}: \Omega \to {\mathbb R}$.
After dropping the tilde in the notation and using $(x,y)$ to denote $(c(x-ct),cy)$, 
the traveling pulse problem is to find $(u,v,c)$ on the rescaled domain $\Omega^* = (-\infty,\infty)\times (-cL,cL)$ satisfying 
\begin{equation}
 \begin{array}{rl}
\ds dc^2 \Delta u +dc^2 \partialx{u} +f(u)-v &=0  , \\
\ds c^2 \Delta {v} +c^2 \partialx{v}+u-\gamma \, v  &=0 
        \end{array}
        \label{eqn:fhn3_2d} 
\end{equation}
with $(u,v)\to (0,0)$ as $|x|\to \infty$ and $(u,v)=(0,0)$ on 
$\partial\Omega^*$.  

We introduce Hilbert spaces $\Lex({\Omega^*})$ and $\Hex({\Omega^*})$, corresponding to 
the inner products 
\begin{equation}
\begin{aligned}
\langle v,w\rangle_{\Lex} &\equiv \int_{\Omega^*} e^x  vw \, dx \, dy \\
\text{and} \ \ \langle v,w\rangle_{\Hex} &\equiv \int_{\Omega^*} e^x \left\{ \nabla{v}\cdot \nabla{w} +vw  \right\}\, dx \, dy , 
\end{aligned}
\label{eqn:innerProd_2d}
\end{equation}
respectively.  The induced norms are again denoted by $\| \cdot \|_{\Lex}$ and
$\| \cdot \|_{\Hex}$.  Let $W$ denote the subspace 
\begin{equation}
W\equiv \left\{ v\in \Hex\, |\, T_0 (v)=0 \right\}, 
\label{eqn:W}
\end{equation}
where $T_0$ is the trace operator on $\Hex (\Omega^*)$.  
The variational approach of Section~\ref{sec:vf} can be extended now to 
find weak solutions $(u,v)\in W\times W$ to ~\eqref{eqn:fhn3_2d}. 
As in the 1D case,  we write $v=\Lc u$ and the 
functional $J_c:\Hex\to\mathbb{R}$ is defined as 
\begin{equation} 
J_c (w) \equiv 
\int_{\Omega^*} e^x \left\{ \frac{dc^2}{2} \left| \nabla w\right|^{\,2} 
+ \frac{1}{2} w\, \Lc w +\F (w) \right\}\, dx \, dy .
\label{eqn:Jc_2d} 
\end{equation}

We restrict our attention to the following admissible set 
in order to avoid a continuum of solutions due to  translation: 
\begin{equation}
\A \equiv \left\{ v\in W\, : \, \| v\|_{\Hex}^2 =2 \right\}.
\label{eqn:A_2d}
\end{equation}
Suppose $u_c$ is a minimizer of $J_c$ in the admissible set ${\cal A}$, 
the traveling wave speed $c_0$ will be determined by $J_{c_0}(u_{c_0})=0$, and
$(u_{c_0}, v_{c_0}, c_0)$ is a traveling wave solution.  

We follow along in Sections~\ref{sec:vf}-\ref{sec:steepestDescent} and find that 
the equation~\eqref{eqn:wStar} for the update $w^*=\lambda q +(\mu +dc^2)w$ 
still holds in two dimensions, if one interprets $w^*\in W$ and the Laplacian 
operator correctly.  The operators $\Q:\Hex\to\Hex$ and $s:\Hex\to\Hex$ may be 
extended in a trivial way, so that Algorithm~\ref{alg:alg1} may still be thought 
to hold in two dimensions, so long as we do not use the clipping operator.  
Equivalently, define $r(w)=w$ for all $w\in \Hex$ to be the identity for our 
computations in two dimensions.  

\begin{remark} \label{remark_scale}
In theoretical studies, it may be more convenient to employ the scaling $(X,Y)=(c\xi,y)$ so that the transformed domain $\Omega^*$ is the same
as $\Omega$. For numerical computation, so long as we have the same number of mesh points in the vertical direction, there is essentially very little 
difference between the two scalings.
\end{remark}


\section{Computation of pulses in one dimension}\label{sec:results_1d} 
We demonstrate how our method can be used to calculate traveling pulse 
solutions in the class~\MPM for one dimension in space,  allowing for the 
facts that the subclasses $+$, $-$, $-/+$ and $+/-$ are all subsets 
of~\MPM.  
The steepest descent algorithm will continue to work even if the iterates 
degenerate into functions in these subclasses. However, in our experiments 
we do not observe this to happen. 
{In~\secref{sec:varyCD} we} will investigate two aspects of the 
theory regarding traveling waves: the possibility of multiple traveling 
pulse solutions and the validation of the asymptotic 
relation~\eqref{eqn:fastSpeed}.  In both cases, the value of the parameter 
$d$ is important.  In~\cite[Theorem 1.1]{CC2015}, it is shown that when $d$ 
is small a traveling pulse solution must exist, but it is not known what 
happens for larger $d$ or if multiple pulse solutions may exist for a 
particular $d$.  The dependence~\eqref{eqn:fastSpeed} holds only for the 
fastest traveling pulse solution. 

In~\secref{sec:parabolicTest}, the computed traveling pulses are tested 
using a parabolic solver.  
{It will be helpful to distinguish between the meaning of the space 
variable $x$ in~\eqref{eqn:fhn1} versus the variable that represents space 
for~\eqref{eqn:fhn3} (up to a shift), let us say $z=cx$.  Hereafter, $z$ shall 
denote the space variable used in~\algref{alg:alg1}.  When we study the 
results using the parabolic solver or when we consider our traveling pulses 
as solutions of~\eqref{eqn:fhn1}, we instead use the variable $x$ for space.} 

\subsection{Numerical methods for~\algref{alg:alg1}}\label{sec:numerics_1d} 
Some {specific} numerical methods must be {adopted} for the computations 
in~\algref{alg:alg1}.  We do not seek to compare {different implementations}.  The goal 
is to investigate our algorithm assuming that each step is performed with 
reasonable accuracy, for which purpose there are myriad acceptable 
numerical methods.  Spatial discretizations are performed 
with standard, centered finite difference methods that are formally 
second-order accurate with respect to the uniform grid size, $h>0$.  
Numerical integrations were computed using the {composite} midpoint rule.
Shifting operations were handled by shifting the grid points themselves, 
rather than interpolating the shifted data onto a fixed grid.  This is easy 
to implement and avoids introducing interpolation errors at each step of 
the algorithm.  

Let $\Omega = (a,b)$ for $-\infty <a\ll 0 \ll b <\infty$ denote the 
domain (which may be shifted each iteration) and denote the computational 
grid points by $x_j = a+hj$, for $j=0,1,\ldots, N+1$.  Here, $(N+1)h=b-a$.  
In order to make some precise statements regarding our computations below, 
any functions, say $\psi (x)$, defined at the grid points will have 
approximate values $\psi_j\approx \psi (x_j)$, $0\leq j\leq N+1$.  
Due to the computational truncation of the domain, asymptotic boundary 
conditions are implemented; see~\appref{sec:asymbc} for details. 

For the parabolic solver (used to test stability), the same discrete 
methods are applied in space as described above, with 
Crank-Nicolson for the time evolution.  Newton's method is used for the 
nonlinearity.  However, the $u$ and $v$ computations are not done at the 
same time levels.  Rather, they are staggered by half a time step in order 
to numerically decouple their calculations, for efficiency.  The resulting 
method is formally of second-order accuracy in both space and time.  The 
idea of staggered 
space and time methods has been used often since the seminal work 
of~\cite{VNR1950}.  Unlike this early work with
hyperbolic conservation laws, our parabolic solver does not require
a staggered grid in space for stability.  Also, we move the grid points 
each time step in accordance with the calculated wave speed, in order to 
avoid using a very large domain to test the wave propagation over long 
times.  For this scheme, the wave profile would ideally appear to be static.  
If the computed profile or speed is not correct, it will appear as a 
deviation from the initial profile used in the solver.  

\subsection{Traveling pulse behavior for various values of $d$.}\label{sec:varyCD}
Since a traveling pulse solution of~\eqref{eqn:fhn1} corresponds to a root of 
$\Jc$, we first investigate the dependence of $\Jc$ on $c$ and $d$.  
We know from~\secref{sec:monotone} that $\Jc$ increases with $d$; this 
qualitative result, in particular \eqref{eq:orderJ},  helps us to 
quantitatively locate the right ranges of $d$ and $c$. 

With a fixed $d$, 
the values of $\Jc$ are calculated for various values of $c$ 
using~\algref{alg:alg1} until we find where $\Jc$ changes sign, thus signifying 
a root.  A good approximation is then calculated for the wave speed $c$, where 
$\Jc=0$, by the method of {\em regula falsi} (see \eg~\cite{SB2010}).  This is 
a root-finding algorithm that approximates the wave speed by using linear 
interpolation across the interval where $\Jc$ changes sign.  Upon completion, 
we also obtain the traveling pulse profile from~\algref{alg:alg1}.

We begin by providing examples of the curves $\Jc$ for 
$d=2e-3$, $d=5e-4$, $d=3e-4$ and $d=1e-4$.  
The results are shown in~\figref{fig:results1} at a resolution of $100$ 
evenly-spaced samples of the wave speed per unit of $c$.  The computational 
grid spacing is $h=10^{-2}$.  Other parameters 
values in~\algref{alg:alg1} are provided 
in~\tabref{tab:tab1}.  For the relative accuracy tolerance $\delta_1$ and 
domain width $|\Omega|$ {(for spatial variable $z$)}, the values were 
adjusted within the ranges shown experimentally, depending on $c$.  
\begin{table}[h]
\centering
\begin{tabular}[c]{|c|c|c|c|c|c|c|c|}\hline 
$\beta$ & $\gamma$ & $\theta$ & $\alpha^0$ & $\delta_1$ & $\delta_2$ & $\delta_3$ & $|\Omega|$  \\ \hline
$1/4$ & $1/16$ & $1/2$ & $1/1000$ & $10^{-9}$--$10^{-7}$ & $10^{-14}$ & $10^{-3}$ & $160$--$480$ \\ 
\hline 
\end{tabular}
\caption{Parameter values for the tests corresponding to~\figref{fig:results1}.\label{tab:tab1}}
\end{table}

From Step 4 of~\algref{alg:alg1}, the numbers $Q^n-(dc^2+\mu^n) w^n$ 
represent the update to the wave profile at iteration $n$.  The sizes of these 
values were found to vary with both $c$ and $n$ significantly.  In order to 
define a rule for the step sizes $\alpha^n$, we introduced a normalization, 
first by writing 
$\alpha^n=\alpha_1^n/\alpha_2^n$ and then choosing 
\[
\alpha_2^n \equiv \max_{0\leq j \leq N+1} \left| Q^n_j-(dc^2+\mu^n) w^n_j \right| . 
\]
The value of $\alpha_1^n$ was allowed to either increase or decrease for 
purposes of Step 6 in~\algref{alg:alg1}.  We applied the rule 
$\alpha_1^{n+1} =\min \{ 1.1 \alpha_1^n , \alpha_1^0 \}$.  
It was observed that this approach sped up convergence compared to 
taking a constant $\alpha_1^n = \alpha_1^0$.  Finally, the initial profile used 
at $c=5$ was 
\begin{equation*}
w^0 (x) = \left\{ \begin{array}{cc}
1,& \ \text{if} \ -1 \leq x \leq 1, \\
0, & \ \text{otherwise} . 
\end{array}\right.
\end{equation*}
This square-wave is independent of $c_0$ when using the $x$-coordinate, but 
when used in~\algref{alg:alg1}, we must rescale first by $z=c x$, so that 
\begin{equation}
w^0 (z) = \left\{ \begin{array}{cc}
1,& \ \text{if} \ -5 \leq z \leq 5, \\
0, & \ \text{otherwise} . 
\end{array}\right.
\label{eqn:initialGuess}
\end{equation}
A step size $\Delta c$ with $|\Delta c|=0.01$ was used to change the wave speed to 
$c+\Delta c$. The value of $\mathcal{J} (c+\Delta c)$ was then calculated using 
the computed minimizer corresponding to $\Jc$ as the initial guess 
{$w^0 (z)$}.  This process was repeated to generate our results.  

\begin{figure}
\includegraphics[width=\textwidth]{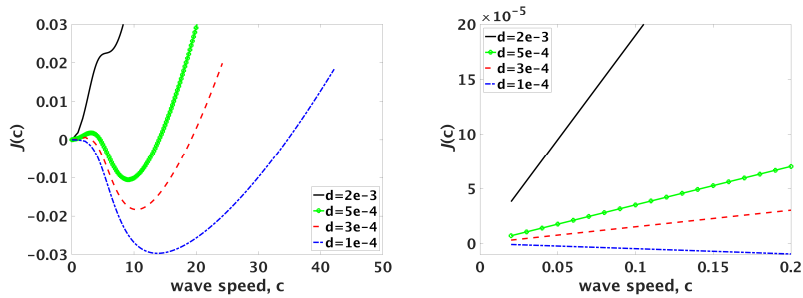}
\caption{The computed values of $\Jc$.  Left: note the existence of multiple 
traveling pulse solutions for $d=3e-4$ and $d=5e-4$.  Right: note that the 
curves do not cross the horizontal axis near $c=0$.  For small values of $d$, 
the traveling pulse solution might be unique in the class~\MPM.  {For large 
values of $d$, there is no traveling pulse solution.}}  
\label{fig:results1}
\end{figure}

In~\figref{fig:results1}, note that $\Jc$ crosses the horizontal axis twice 
when $d=5e-4$ and $d=3e-4$, but only once in case $d=1e-4$.  {Also, we see 
that for large enough values of $d$ there is no traveling pulse solution in the 
class $\MPM$.}  In case of 
multiple roots of $\Jc$, let $c_1=c_1 (d)$ and $c_0=c_0 (d)$ denote the smaller 
and larger of the wave speeds, respectively, such that $\Jc=0$.  Our 
computations suggest that there exists some critical value $d_{crit}$ such that 
\[
\lim_{d\to d_{crit}^+} c_1(d) =0.
\]
Furthermore, for $d<d_{crit}$ the traveling wave solution might be unique in 
the class~\MPM.  Stability of the traveling pulses is discussed below 
in~\secref{sec:parabolicTest}.

In~\tabref{tab:tab2} we have provided the computed values of the ratio 
\[
\eta \equiv	\frac{2dc_0^2}{(1-2\beta)^2}. 
\]
In accordance with~\eqref{eqn:fastSpeed}, as the value of $d$ decreases we see 
the ratio $\eta$ becomes closer to the limiting value of $\eta=1$.  

\begin{table}[h]
\centering
\begin{tabular}[h]{|c|c|c|c|}\hline 
$d$ & $c_1$ & $c_0$ & $\eta$ \\ \hline
$5e-4$ & $4.58$ & $14.04$ & $0.79$ \\ \hline
$3e-4$ & $3.14$ & $19.18$ & $0.88$ \\ \hline
$1e-4$ & --- & $34.70$ & $0.96$ \\ 
\hline 
\end{tabular}
\caption{Computed values of the wave speeds $c_0$ and $c_1$ for traveling pulse solutions.\label{tab:tab2}}
\end{table}

\subsection{Testing of candidate traveling {pulse} profiles}\label{sec:parabolicTest}
Our computed traveling pulses are tested here via the method 
in~\secref{sec:numerics_1d}.  That is, the computed wave profiles,
after scaling back {from spatial variables $z$ to $x$}, are 
used to initiate a parabolic solver and the computational window is moved at a 
rate of one grid length per time step.  The grid length, $h$, and time step size, $\Delta t$, are 
related by $c\Delta t=h$ with $c$ the computed speeds in~\tabref{tab:tab2}, so 
that the wave profile should move at the same rate as the computational 
window.  A stable profile should then retain its shape and speed for a long 
time.  
The calculations were run for each case until either the profile was observed 
to break up or until the pulse propagated the distance of one computational domain 
length.  

For the slower wave profiles with speed $c_1$, our tests showed that the 
profiles were not maintained by the parabolic solver, nor did they evolve into 
a new traveling pulse profile.  Indeed, the profiles broke up completely after 
a relatively small time.  An example is shown for $d=5e-4$ 
in~\figref{fig:fig0} ({top}).  The slower traveling pulse 
solutions that exist for larger values of $d$ are unstable.  

For the faster wave speeds denoted by $c_0$, we observed that the 
wave profiles remained stable.  In~\figref{fig:results2} 
the initial data is plotted together with the final data, upon completion of the 
parabolic solver run.  The final data 
is shifted left in space by one domain length for a 
direct, visual comparison with the initial data.  The profiles have 
traveled the length of their computational domain and retained their shape 
and speed very well in all cases.  
\begin{figure}
\includegraphics[width=\textwidth]{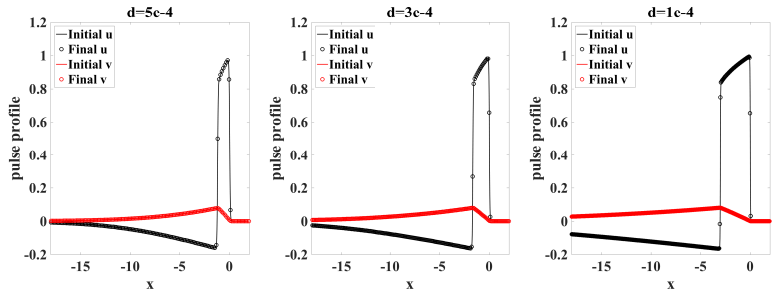}
\caption{Computed pulse profiles, shifted back 
in space by one domain length for comparison with the initial profiles, for 
cases $d=5e-4$ (left), $d=3e-4$ (middle) and $d=1e-4$ (right).  Circle 
markers are shown $200$ grid points apart for the final profiles.  The initial 
and final profiles are visually identical, indicating stability.}  
\label{fig:results2}
\end{figure}

\section{Computation of pulses in two dimensions}\label{sec:results_2d} 
In this section we demonstrate the calculation of traveling (dissipative) solitons 
in a multiple-dimensional domain.  Partial curves $\Jc$ are shown for four
values of $d$, illustrating cases where there are multiple, one or no 
solutions found, including solitons and unstable pulses.  
A fundamental difference between the one- and two-dimensional cases 
relates to an occurrence of bifurcation in the latter case, wherein a soliton is observed to split 
into two co-evolving spots {in the direction perpendicular to their motion}.  

We take $\Omega = (-\infty,\infty)\times (-L,L)$, with $L=1$ in all cases.  
As discussed in~\secref{sec:alg2D}, we rescale the domain to 
$\Omega^* = (-\infty,\infty)\times (-cL,cL)$; the truncated, computational 
domain for~\algref{alg:alg1} is then chosen to be $[a,b]\times[-c,c]$.  
Recall that the domain is shifted during iterations of~\algref{alg:alg1}, so we only specify the domain length $b-a$ below.  Let $(x,y)$ and 
$(x^*,y^*)$ denote elements of $\Omega$ and $\Omega^*$, respectively. 

Numerical methods to compute pulse profiles and for numerical integration 
are simple extensions of the methods described in~\secref{sec:numerics_1d} 
into two dimensions.  That is, standard second-order, centered 
finite differences were used on a uniform rectangular grid 
to implement~\algref{alg:alg1}. 
Numerical integration was implemented using midpoint approximations on 
each grid rectangle, also with centered, second-order rules to evaluate the 
terms of the integrands.  The parabolic solver for stability tests is also 
analogous to that of~\secref{sec:numerics_1d}, using centered finite 
differences in space and time stepping via Crank-Nicolson.  Other relevant 
parameter values are given in~\tabref{tab:tab4}.  We apply asymptotic 
boundary conditions, as per~\appref{sec:asymbc}.  

\begin{table}[h] 
\centering 
\begin{tabular}[c]{|c|c|c|c|c|c|c|}\hline 
	$\beta$ & $\gamma$ & $\theta$ & $\alpha^0$ & $\delta_1$ & $\delta_2$ & {$\delta_3$} \\ \hline 
$1/4$ & $1/16$ & $1/2$ & $10^{-5}-10^{-3}$ & $10^{-6}$ & $10^{-12}$  & {$10^{-3}-10^{-1}$} \\ 
\hline 
\end{tabular} 
\caption{Parameter values for the tests corresponding to~\figref{fig:fig2d_1}.\label{tab:tab4}}
\end{table} 

The functional values $\Jc$ are shown for three values of $d$ 
in~\figref{fig:fig2d_1}.  For these computations we fixed $b-a=280$, 
with $5600$ grid intervals from $x^*=a$ to $x^*=b$ and $80$ grid intervals from 
$y^*=-c$ to $y^*=c$.  We observe that for $d=5e-4$ and $d=7e-4$ a single 
traveling {soliton} is found; that is, there is a single root of $\Jc$ at 
$c=c_0$.  We calculate $c_0\approx 13.74$ for $d=5e-4$ and $c_0\approx 10.54$ 
for $d=7e-4$.  When $d=9e-4$, the root is lost and no solution is found.  

\begin{figure}
\includegraphics[width=\textwidth]{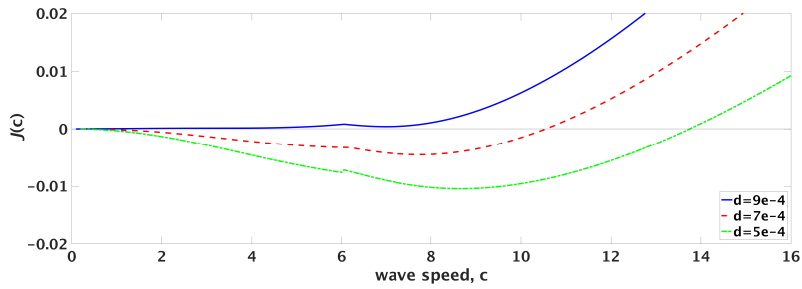}
\caption{The computed values of $\Jc$ in two dimensions, for three choices 
of $d$.  For smaller values of $d$ we find one traveling pulse solution.  For 
$d=9e-4$, $\Jc$ remains above the horizontal axis and we observe no solution.  
A transitional behavior is observed for values of $c$ near $6$, due to the 
splitting of a soliton.}  
\label{fig:fig2d_1}
\end{figure}

The computed {solitons} for {both} cases $d=5e-4$ and $d=7e-4$ were 
{confirmed} to be 
stable upon inserting these as initial conditions for a parabolic solver and 
{allowing the spots to propagate for the length of one computational domain.}  
Upon completion, we shift the {solitons} back to 
the left by the {same distance}, for a direct comparison with the 
initial profile.  
{In~\figref{fig:fig2d_2}, we show that the initial and final profiles are 
visually identical.}  

\begin{figure}
\includegraphics[width=\textwidth]{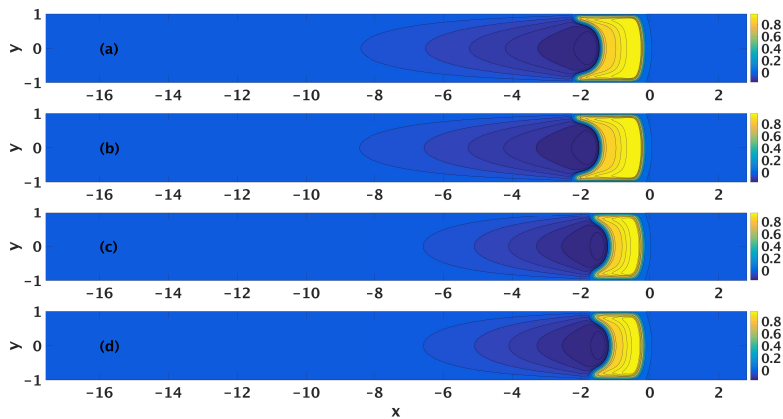}
\caption{Contour plots of $u$.  Plot (a) shows the {soliton} profile, as 
computed using~\algref{alg:alg1} for $d=5e-4$, rescaled to variables $(x,y)$.  
The profile is then used to initiate a parabolic solver, then run until 
{the soliton has propagated one domain length.}  
Plot (b) shows the resulting 
profile, shifted back horizontally 
to compare with the 
initial profile.  For $d=7e-4$, plot (c) shows the initial profile and (d) 
shows the final profile.  The final profiles are visually identical to the 
initial profiles.  These {are} stable; that is, they are traveling 
{dissipative} solitons.}  
\label{fig:fig2d_2}
\end{figure}

A close inspection of~\figref{fig:fig2d_1} reveals that $\Jc$ appears to have 
a local maximum near $c=6$.  This is related to a bifurcation 
that {occurs}. 
{As we track only the minimum energy curve, there may be additional secondary bifurcations associated with
other bifurcation branches, resulting in many close-by solutions.} 
To illustrate this effect clearly, we consider the case of 
$d=8.8e-4$, for which the values of $\Jc$ are plotted in~\figref{fig:fig2d_3}.  
The curve $\Jc$ versus $c$ has two smooth branches, separated near $c\approx 6.1$.
For larger $c$, the minimizer profile has one contiguous positive region.  That is, 
{there is a single soliton.}  As $c$ decreases, a separation 
into two parallel {solitons} serves to reduce the energy.  
This qualitative state persists for the minimizer profile as $c$ decreases 
toward zero.  
As a result, we found four roots of $\Jc$ corresponding to two 
{solitons} and two unstable traveling pulses.  

\begin{figure}
\includegraphics[width=\textwidth]{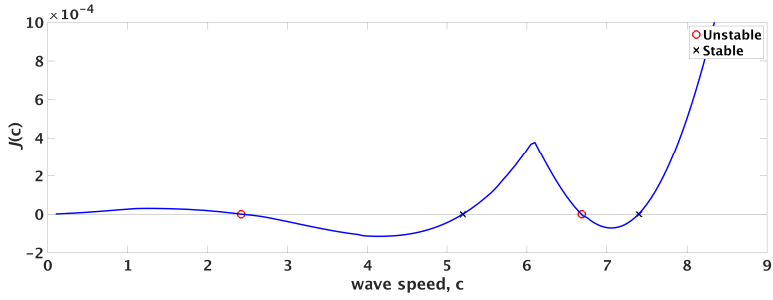}
\caption{The computed values of $\Jc$ for $d=8.8e-4$.  A bifurcation 
occurs near $c\approx 6.1$. As a result, four 
solutions are found; two {solitons} and two are {unstable pulses}.}
\label{fig:fig2d_3}
\end{figure}

We denote the four wave speeds for the traveling pulses by $c_0\approx 7.4067$, 
$c_1\approx 6.6864$, $c_2\approx 5.1752$ and $c_3\approx 2.4181$.  The unstable 
pulses correspond to $c_1$ and $c_3$.  In~\figref{fig:fig2d_4} we show the 
unstable computed profiles $u$, rescaled to variables $(x,y)$, as computed 
by~\algref{alg:alg1}.  

The {stable solitons} correspond to $c_0$ and $c_2$.  In~\figref{fig:fig2d_5}  we show 
these, rescaled to variables $(x,y)$, as computed 
by~\algref{alg:alg1} and also after running the parabolic solver to demonstrate 
stability.  
For visualization purposes, we do not show the entire domain.  We note 
that the computational grid used for these computations with $c> 6.1$, corresponding 
to the single-{soliton} solution, was the same as for other computations in this 
section.  However, for $c<6.1$ the computational domain was shortened so that $b-a=50$, 
with 4000 intervals between $x^*=a$ and $x^*=b$ and $160$ computational intervals 
between $y^*=-c$ and $y^*=c$.  This finer computational grid was needed to approximate 
the two-{soliton} solution.  In~\figref{fig:fig2d_5} we note that the computed {solitons} 
retain their shapes well, but they do not travel with precisely the computed 
wave speeds $c_0$ and $c_2$.  We believe this is simply due to numerical error in 
computing the functional value $J_c (u)$, which could be reduced using a finer 
computational grid or more accurate finite difference method.  
For $d=8.8e-4$, the values of $\Jc$ remain very small over a wide range of $c$ values.  
As a result, it is more difficult to compute the location of the roots of $\Jc$ as 
compared to other examples in this paper.  

\begin{figure}[h!]
\includegraphics[width=\textwidth]{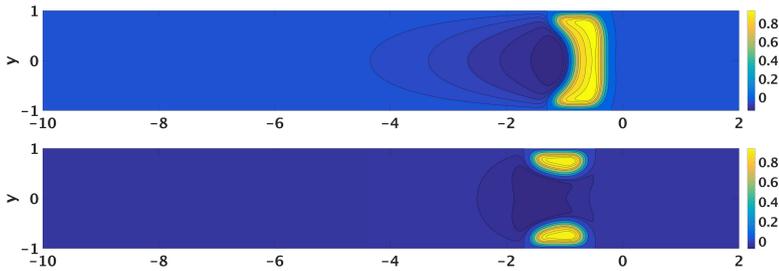}
\caption{Contour plots of the unstable traveling pulse profiles $u$ for $d=8.8e-4$.  
Top: the profile with wave speed $c_1$.  Bottom: the {double}-pulse with 
wave speed $c_3$.} 
\label{fig:fig2d_4}
\end{figure}

\begin{figure}
\includegraphics[width=\textwidth]{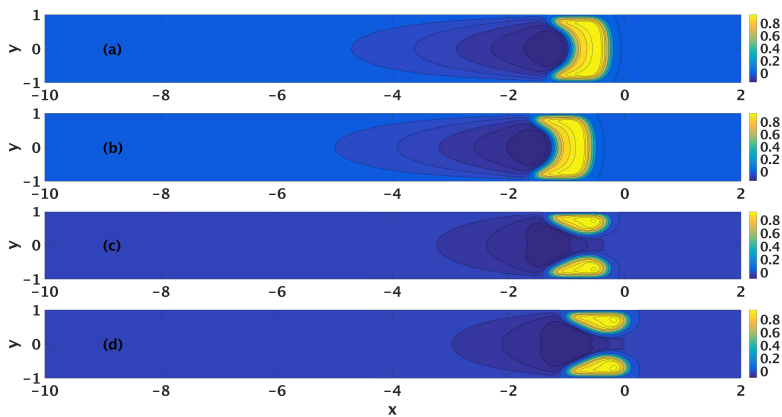}
\caption{Contour plots of the traveling {solitons} for $d=8.8e-4$.  Plot (a) shows the profile {$u$} as 
computed using~\algref{alg:alg1}, rescaled to variables $(x,y)$.  
The profile is then used to initiate a parabolic solver, run until 
{the soliton} propagates one domain length.  Plot (b) shows the 
resulting profile, shifted back horizontally to compare with the 
initial profile.  For the two-{soliton} solutions, plot (c) shows the 
initial profile and (d) shows the final profile, shifted back horizontally.}  
\label{fig:fig2d_5}
\end{figure}

\section{Summary and {future work}}\label{sec:summary} 
We have provided an iterative method to calculate traveling {solitons and 
unstable traveling} pulse solutions for the~\fhn equations.  It is a steepest descent 
method based on the minimization 
of a functional within a certain admissible set.  The infimum of the functional 
over the admissible set, denoted by $\Jc$, depends on a parameter $c$ that 
represents wave speed.  Traveling pulses are identified as roots of the functional; 
$\Jc=0$. We have demonstrated that the method is robust. For example, some tests 
revealed that initial guesses employed for our method would not suffice as initial 
conditions in a parabolic solver to try to compute a {soliton}.  The computations 
also support the asymptotic relationship~\eqref{eqn:fastSpeed} that applies to the 
{solitons (observed as the fastest pulses)}, given a set of physical parameters.  
This provides mutual validation. We {computed} both stable 
and unstable traveling pulses for 
moderate values of the parameter $d$, no traveling pulses for large values of 
$d$ and a unique, stable traveling pulse for small $d$.  
{Solitons} were tested using a parabolic solver and {observed to be 
stable}.  


We also observe that as $d$ becomes small,  the fastest wave speed for the 
soliton becomes large, the pulse width (measured in an 
appropriate sense) becomes wide, and the tail decay rate becomes slow.  
Due to the steep 
wave front but otherwise smooth and slowly-decaying tail, our use of uniform 
grids with finite difference methods is not optimal.  
{This could be addressed through the} 
use of adaptive methods. 
For example, the class of $hp$-adaptive finite element methods have previously 
enjoyed success for problems with a wide range of scales 
(see {\em e.g.}~\cite{DOP1988}).  

In two dimensions of space, we observed a bifurcation 
{that qualitatively separates single and double traveling soliton solutions.}  
The splitting of the {solitons} from one to two {spots} serves to lower 
the functional energy $\Jc$ as $c$ decreases (below around $c\approx 6$ in our 
examples).  For a narrow range of parameter values, this enables $\Jc$ to drop below 
zero multiple times as $c$ changes, resulting in four traveling pulse solutions 
{for a single set of parameters, with their speeds distinct.}  
Two solutions are unstable.  The {two-soliton} solutions have smaller wave speeds than 
the {single-soliton} solutions.  

In some on-going work~\cite{CCD_new}, we will demonstrate how to 
use our algorithm to find traveling fronts as well {in 2D}.  
In fact, for the same physical parameters, fronts and pulses can co-exist.
Our steepest descent method can find many traveling {waves} 
independently {for systems with a variational structure}, but it could also serve as 
a {robust tool to augment the use of 
continuation methods, which may have difficulty in multiple dimensions sometimes 
(see~\secref{sec:introduction}).  Conceivably, one might also use the global 
property to create an {\em ad-hoc} continuation-steepest descent method that can take 
larger steps along a bifurcation curve, saving total computational expense for detailed 
explorations of parameter space.}  


\bibliographystyle{apa}
\bibliography{refs}

\appendix
\section{~\algref{alg:alg1} with asymptotic boundary conditions}\label{sec:asymbc} 
The traveling pulses decay to zero as $|x| \to \infty$ for both one- or 
two-dimensional domains.  
Instead of imposing the zero Dirichlet boundary condition on the bounded 
computational domain in~\algref{alg:alg1}, we 
{derive asymptotic boundary conditions that provide} 
 better information on solution behaviors as $|x| \to \infty$ than just knowing 
that they go to zero.  
 If the governing equation is linear, eliminating the blow-up mode will yield the asymptotic information.
 Similar conclusions can be drawn by linearizing the nonlinear equations about the zero equilibrium point.
Such an idea has been given in, {for example,} \cite{LK1980}.  

Specifically, we derive 
asymptotic boundary conditions to solve for $v=\Lc w$ and $w^*=\Q (w)$  in Steps 1 
and 3 of~\algref{alg:alg1}.  In practice, we have found that the minimum 
computational domain lengths are restricted by these calculations, whereas the 
integrations in Steps 2 and 6 exhibit faster convergence.  This is because as $x\to \infty$, 
 the pulse profiles vanish very quickly, while when $x\to-\infty$ the term $e^x$ 
in the integrands forces the fast convergence of the integrals even though the 
profiles do not vanish as quickly in this direction.  For these 
reasons, we neglect further discussion of errors in integrated quantities due to 
the truncation of the domain.  

\subsection{Computing ${\Lc w}$ and $w^*$ with a given $w$ in one dimension} \label{sec:v_asym}
Suppose a function $w$ is defined on 
the real line $(-\infty,\infty)$; 
however it is known only on the interval $[a,b]$. 
First, we will construct asymptotic boundary conditions for Step 1 in~\algref{alg:alg1}. 
 As $w$ serves as a guess of the minimizer $u$ which decays to zero at infinity,
we assume $w$ and $w'$ are $o(1)$ outside the interval $[a,b]$. 
 Let $v=\Lc w$. Then $v''+v'- \frac{\gamma}{c^2}v= -\frac{w}{c^2}$
on $(-\infty,\infty)$,
which is equivalent to the system
\begin{equation} \label{eqn:v_sys}
\vectwo{v}{z}'=B \vectwo{v}{z} - \vectwo{0}{\frac{w}{c^2}}
\end{equation}
where $B=\mattwo{0}{1}{\frac{\gamma}{c^2}}{-1}$. The eigenvalues of $B$ are given by
\begin{equation} \label{eqn:nu}
\left\{ \nu_1\, ,\, \nu_2 \right\}= \left\{\frac{1}{2} \left(-1 - \sqrt{1+ \frac{4 \gamma}{c^2}}\right)\, , \, \frac{1}{2} \left(-1 + \sqrt{1+ \frac{4 \gamma}{c^2}}\right)\right\}, 
\end{equation}
with $\nu_1<-1<0<\nu_2$. Correspondingly, ${\bf L}_1= \vectwo{-\nu_2}{1}$ and ${\bf L}_2 =\vectwo{-\nu_1}{1}$
are the left eigenvectors of $B$ for $\nu_1$ and $\nu_2$, respectively.
By taking the scalar product of ${\bf L}_1$ with \eqref{eqn:v_sys}, we obtain
\begin{equation} \label{eqn:Phi1}
\Phi_1'= \nu_1 \Phi_1 - \frac{w}{c^2} 
\end{equation}
where $\Phi_1 \equiv {\bf L}_1 \cdot \vectwo{v}{z}=- \nu_2 v+z$. This first order equation can be integrated to give
\[
\Phi_1(x) = - \int_{-\infty}^x e^{\nu_1(x-t)} \frac{w(t)}{c^2} dt ;
\]
the arbitrary constant associated with the complementary solution has to be set to zero for $\Phi_1$
to stay bounded as $x \to -\infty$.
It follows that
\begin{eqnarray*}
\nu_1 \Phi_1(a)-\frac{w(a)}{c^2} &=& \frac{\nu_1}{c^2} \int_{-\infty}^a e^{\nu_1(a-t)} (w(a)-w(t)) \, dt \\
&=& - \frac{1}{c^2} \int_{-\infty}^a e^{\nu_1(a-t)} w'(t) \, dt.
\end{eqnarray*}
Hence
\begin{eqnarray*}
\left|\nu_1 \Phi_1(a)-\frac{w(a)}{c^2} \right| &\leq & \frac{|o(1)|}{c^2} \int_{-\infty}^a e^{\nu_1(a-t)}  \, dt. \\
& = & \frac{|o(1)|}{|\nu_1|\,  c^2} \;.
\end{eqnarray*}
It is therefore natural to impose the boundary condition $\nu_1 \Phi_1=\frac{w}{c^2}$ at $x=a$; which
amounts to
\begin{equation} \label{eqn:left_bc}
v'- \nu_2 v= \frac{w}{\nu_1 c^2} \quad \mbox{at} \;\; x=a.
\end{equation} 
This is like setting the right hand side of \eqref{eqn:Phi1} to zero at $x=a$.
A similar analysis for large positive $x$ using $\Phi_2= {\bf L}_2 \cdot \vectwo{v}{z}$ leads to 
\begin{equation} \label{eqn:right_bc}
v'- \nu_1 v= \frac{w}{\nu_2 c^2} \quad \mbox{at} \;\; x=b.
\end{equation}
\eqref{eqn:left_bc} and \eqref{eqn:right_bc} are the asymptotic boundary conditions 
used when solving for $v=\Lc w$.

\begin{remark} \label{remark4}
If $w$ goes to different constants as $x \to  \pm \infty$ but with $w'=o(1)$ beyond $[a,b]$, the above argument
can be modified to derive some different asymptotic boundary conditions. 
This observation will have implications in case one studies a traveling front problem numerically.
\end{remark}

We will now compute $w^*$ from~\eqref{eqn:wStar} in Step 3 of~\algref{alg:alg1} 
using asymptotic boundary conditions. Let $\hat{w}=dc^2 w -\Lc w +f(w)$ denote
the known right hand side of~\eqref{eqn:wStar}.
 Compare this problem with the equation on $v$ in~\secref{sec:v_asym}. By substituting $\gamma/c^2$ by $1$
 and $w/c^2$ by $\hat{w}$, the new eigenvalues now are $\nu_1^*=-\frac{1}{2}(1+\sqrt{5})$ and 
 $\nu_2^*=\frac{1}{2}(\sqrt{5}-1)$, and the asymptotic boundary conditions are given by
 \begin{eqnarray}
{w^*}{\,'}- \nu_2^* w^* &=  \frac{\hat{w}}{\nu_1^*} & \mbox{at} \;\; x=a, \label{eqn:*left_bc} \\
{w^*}{\,'}- \nu_1^* w^* & = \frac{\hat{w}}{\nu_2^*} & \mbox{at} \;\; x=b. \label{eqn:*right_bc}
\end{eqnarray} 

\subsection{Computing ${\Lc w}$ and $w^*$ with a given $w$ in two dimensions} \label{sec:v_asym2D}
We take $w=w(x,y)$ on the infinite strip $(-\infty,\infty)\times [-L,L]$ 
and derive asymptotic boundary conditions to apply on the truncated domain 
$\Omega=[a,b]\times [-L,L]$, first for $v=\Lc w$.  At $y=-L$ and $y=L$ 
the boundary values are $w=v=0$, for all $x\in\mathbb{R}$.  Fourier 
expansions for $v=\Lc w$ and $w$ are 
\begin{equation}
\begin{aligned}
w(x,y) &= \sum_{j=1}^\infty \hat{w}_j (x) \sin \left(\frac{j\pi (y+L)}{2L}\right) , \\ 
v(x,y) &= \sum_{j=1}^\infty \hat{v}_j (x) \sin \left(\frac{j\pi (y+L)}{2L}\right) .
\end{aligned}
\label{eqn:fourier}
\end{equation}
Insert the relations~\eqref{eqn:fourier} into the equation 
$\Delta v +v_x -\frac{\gamma}{c^2} v =-\frac{1}{c^2} w$: 
\begin{multline*}
\sum_{j=1}^\infty \left( \hat{v}_j'' (x)+\hat{v}_j' (x)-\left(\frac{j^2\pi^2}{4L^2}+\frac{\gamma}{c^2}\right) \hat{v}_j (x) \right) \sin \left(\frac{j\pi (y+L)}{2L}\right)
= \\ 
\sum_{j=1}^\infty -\frac{1}{c^2}\hat{w}_j (x) \sin \left(\frac{j\pi (y+L)}{2L}\right) .
\end{multline*}
Then the Fourier coeffcients satisfy 
\begin{equation}
\hat{v}_j'' (x)+\hat{v}_j' (x)-\left(\frac{j^2\pi^2}{4L^2}+\frac{\gamma}{c^2}\right) \hat{v}_j (x) 
= -\frac{1}{c^2}\hat{w}_j (x) .
\label{eqn:bc2d1}
\end{equation}

In case $x<0$ with $|x|\gg 1$, we assume it holds that 
$|\hat{w}_j (x)|\ll |\hat{w}_1 (x)|$ and 
$|\hat{v}_j (x)|\ll |\hat{v}_1 (x)|$ for all $j>1$, thus
\begin{equation}
\begin{aligned}
w(x,y) &\approx \hat{w}_1 (x) \sin \left(\frac{\pi (y+L)}{2L}\right) , \\ 
v(x,y) &\approx \hat{v}_1 (x) \sin \left(\frac{\pi (y+L)}{2L}\right) .
\end{aligned}
\label{eqn:fourier2}
\end{equation}

Then $\hat{w}_1 (x)$ and $\hat{v}_1 (x)$ have the same behavior in $x$ as 
$w$ and $v$, respectively, as $x\to-\infty$.  Furthermore, these Fourier 
coefficients satisfy~\eqref{eqn:bc2d1}.  By analogy with the derivation 
in~\secref{sec:v_asym}, if $a<0$ with $|a|\gg 1$ then we apply the 
boundary condition 
\begin{equation} \label{eqn:l_bc2d}
\hat{v}_1'- \nu_2 \hat{v}_1= \frac{\hat{w}_1}{\nu_1 c^2} \quad \mbox{at} \;\; x=a 
\end{equation} 
with the eigenvalues 
\begin{equation}
\begin{aligned}
\nu_1 &= \frac{1}{2} \left(-1 - \sqrt{1+\frac{\pi^2}{L^2}+ \frac{4 \gamma}{c^2}}\right) \\ 
\mbox{and} \;\; 
\nu_2 &= \frac{1}{2} \left(-1 + \sqrt{1+\frac{\pi^2}{L^2}+ \frac{4 \gamma}{c^2}}\right) .
\end{aligned}
\label{eqn:nu_2d}
\end{equation}
We combine~\eqref{eqn:fourier2} and~\eqref{eqn:l_bc2d} to derive the 
approximate boundary condition 
\begin{equation} \label{eqn:left_bc2d}
\frac{\partial v}{\partial x}- \nu_2 v= \frac{w}{\nu_1 c^2} \quad \mbox{at} \;\; x=a , \quad 
\mbox{for} \;\; -L < y < L.
\end{equation} 
It is equivalent to applying~\eqref{eqn:left_bc} at $x=a$ for each fixed 
value of $y$, with the adjustment~\eqref{eqn:nu_2d} for the 
eigenvalues~\eqref{eqn:nu}.  
The corresponding boundary condition on the right is 
\begin{equation} \label{eqn:right_bc2d}
\frac{\partial v}{\partial x} - \nu_1 v= \frac{w}{\nu_2 c^2} \quad \mbox{at} \;\; x=b , \quad 
\mbox{for} \;\; -L < y < L, 
\end{equation} 
by analogy with~\eqref{eqn:right_bc} and the derivation 
of~\eqref{eqn:left_bc2d}.  
Here, we assume $b\gg 1$.  Taken together with $v(x,-L)=v(x,L)=0$ for all 
$x\in\mathbb{R}$,~\eqref{eqn:left_bc2d}-\eqref{eqn:right_bc2d} are the 
boundary conditions used to compute $v=\Lc w$ on the truncated domain.  
Note that the asymptotic 
conditions~\eqref{eqn:left_bc2d}-\eqref{eqn:right_bc2d} are compatible with 
the homogeneous Dirichlet boundary conditions for $v$ and $w$ at $y=\pm L$.  

Asymptotic boundary conditions for $w^*$ in Step 3 of~\algref{alg:alg1} 
may be derived quickly by first comparing~\eqref{eqn:wStar} to the 
equation~\eqref{eqn:fhn3_2d} for $v=\Lc w$.  Let 
$\tilde{w}=dc^2 w -\Lc w +f(w)$ denote the known right hand side 
of~\eqref{eqn:wStar}.  By substituting $\gamma/c^2$ with $1$ and $w/c^2$ 
with $\tilde{w}$, the new eigenvalues now are 
\begin{equation}
\left\{ \nu^*_1 \, ,\, \nu^*_2 \right\} = \left\{ \frac{1}{2} \left(-1 - \sqrt{5+\frac{\pi^2}{L^2}}\right) \, , \, 
\frac{1}{2} \left(-1 + \sqrt{5+\frac{\pi^2}{L^2}}\right) \right\}, 
\label{eqn:nuStar_2d}
\end{equation}
and the asymptotic boundary conditions are given by
 \begin{eqnarray}
\frac{\partial w^*}{\partial x}- \nu_2^* w^* &=  \frac{1}{\nu_1^*}\tilde{w} & \mbox{at} \;\; x=a, \;\; -L<y<L, \label{eqn:*left_bc2d} \\
\frac{\partial w^*}{\partial x}- \nu_1^* w^* & = \frac{1}{\nu_2^*}\tilde{w} & \mbox{at} \;\; x=b,  \;\; -L<y<L. \label{eqn:*right_bc2d}
\end{eqnarray}

\end{document}